\documentclass[10pt,twoside]{article}
\usepackage{amssymb,amsmath,enumerate,amsthm,verbatim,graphicx,tikz}
\usepackage{url}
\usepackage{makeidx}
\usepackage[hmarginratio=2:2]{geometry}
\usepackage[all]{xy}
\usepackage{fancyhdr}

\AtEndDocument{\bigskip{\footnotesize
  \textsc{Mathematics Institute, University of Warwick, Coventry, UK} \par  
  \textit{E-mail address}, A.~Koutsianas: \texttt{a.koutsianas@warwick.ac.uk}
}}

\newtheorem{lemma}{Lemma}[section]
\newtheorem{proposition}{Proposition}[section]
\newtheorem{theorem}{Theorem}[section]
\newtheorem{definition}{Definition}[section]

\newcommand{\ZZ}{\mathbb Z}
\newcommand{\NN}{\mathbb N}
\newcommand{\pK}{\mathfrak p}

\newcommand{\bL}{\mathfrak B}
\newcommand{\QQ}{\mathbb Q}
\newcommand{\OK}{\mathcal O_K}

\newcommand{\ft}{f^{\tau}}
\newcommand{\Lt}{L^{\tau}}
\newcommand{\SiK}{\mathcal O_{K,S}}

\newcommand{\SuK}{\mathcal O^*_{K,S}}

\newcommand{\SuL}{\mathcal O^*_{L,S_{L}}}

\newcommand{\SuLKo}{\mathcal O^*_{L,K,S_L,1}}
\newcommand{\SuLKpmo}{\mathcal O^*_{L,K,S_L,\pm 1}}
\newcommand{\pgl}{\text{PGL}_2(\ZZ)}
\newcommand{\Gl}{G_\lambda}
\newcommand{\Gm}{G_\mu}
\newcommand{\Sl}{S_\lambda}
\newcommand{\Sm}{S_\mu}
\newcommand{\Slinf}{S_\lambda^\infty}
\newcommand{\Sminf}{S_\mu^\infty}
\newcommand{\cinf}{c_{1,\infty}}

\newcommand{\lI}{\lambda_I}
\newcommand{\linf}{\lambda_\infty}
\newcommand{\minf}{\mu_\infty}

\newcommand{\ord}{\operatorname{ord}}

\newcommand{\Gal}{\operatorname{Gal}}
\newcommand{\rank}{\operatorname{rank}}

\newcommand{\Norm}{\operatorname{Norm}}

\begin{document}

\title{Computing all elliptic curves over an arbitrary number field with prescribed primes of bad 
reduction}
		
\author{Angelos Koutsianas}
%\address{Mathematics Institute\\
%University of Warwick\\
%Coventry\\
%United Kingdom}
%\email[A.~Koutsianas]{a.koutsianas@warwick.ac.uk}
\date{}%Last update: \today}
\maketitle

\begin{abstract}
In this paper we study the problem of how to determine all elliptic curves defined over an arbitrary 
number field $K$ with good reduction outside a given finite set of primes $S$ of $K$ by solving 
$S$--unit  equations. We give examples of elliptic curves over $\QQ$ and quadratic fields.
\end{abstract}

\section{Introduction}\label{Introduction}

Let $K$ be a number field and $S$ a finite set of primes(non--archimedean) of $K$. By a classical 
result of Shafarevich (see \cite{SilvermanArithmetic}) we know that there are finitely many 
isomorphism classes of elliptic curves defined over $K$ with good reduction at all primes outside 
$S$. Many people have previously discovered methods to find explicit representatives of each 
isomorphism class(\cite{CremonaLingham07}, \cite{Kida01-2}, \cite{Kida01-1}). It is not hard to show 
(see \cite{CremonaLingham07}) that it is enough to determine the $j$--invariants of the isomorphic 
classes and from them we can easily determine the elliptic curves. In the Cremona--Lingham 
method, given in \cite{CremonaLingham07}, the $j$--invariants are found by computing $S$--integral 
points on specific elliptic curves. However, the available method of finding $S$--integral on an 
elliptic curve requires the Mordell--Weil group of the curve and that problem may be really hard.

The new idea\footnote{Suggested to John Cremona in a personal communication with Noam Elkies in June 
2010.}, which we display here, is to find the possible $j$--invariant indirectly using the 
Legendre $\lambda$--invariant. They are related by,
\begin{equation}\label{j_lambda}
j=2^8\frac{(\lambda^2-\lambda+1)^3}{\lambda^2(1-\lambda)^2}
\end{equation}
and we know that $\lambda$ lies in the 2--division field $L$ of the associated elliptic curve. Such a 
field $L$ is a Galois extension of $K$ with $\Gal(L/K)$ isomorphic to a subgroup of $S_3$ and 
unramified outside $S^{(2)}=S\cup\lbrace\pK\subset\OK :\pK\mid 2\rbrace$. We will first prove that 
there are only finitely many extensions $L$ with these two properties and will give an algorithm for 
determining all of them. If $S_L$ is the set of all primes in $L$ which are above the primes of 
$S^{(2)}$, then we know that $\lambda$ is in $\SuL$, the group of the $S_L$--units of $L$. In 
addition, $\mu=1-\lambda$ is related to the same $j$ and lies in $\SuL$ hence as well. So, we have 
that a possible $\lambda$ with respect to $j$ satisfies the equation: 
\begin{equation}\label{SunitEquation}
\lambda+\mu=1
\end{equation} 
where $\lambda,\mu\in\SuL$. Solving the above equation for each $L/K$ we find all $\lambda$. Then we 
determine the $j$--invariants of the elliptic curves we are seeking and from these it is easy to find 
the elliptic curves(see \cite[$\S 3$]{CremonaLingham07}).

The above equation (\ref{SunitEquation}) is a special case of the general family of Diophantine 
equations which are called $S$--unit equations\cite{Smartbook}. By the theory of linear forms in 
logarithms and Siegel's theorem (\cite{SilvermanArithmetic}) we know that such a Diophantine equation 
has a finite set of solutions. During the 1980's and 1990's many methods of solving $S$--unit 
equations algorithmically were developed. De Weger in his thesis and an article 
(\cite{DeWegerthesis}, \cite{DeWeger87}) gave an algorithmic solution for the special case when 
$K=\QQ$ using lattice basis reduction algorithms. Many others used and extended De Weger's idea to 
solve $S$--unit equations over an arbitrary number field (\cite{TzanakisDeWeger89}, \cite{Smart95}, 
\cite{TzanakisDeWeger92}, \cite{Smartbook}). However, when the degree of the number field or the set 
of primes $S$ is large then the final sieve step of the method may not be efficient. For this reason, 
a new idea was introduced by Wildanger and was extended by Smart(\cite{Wildanger00}, \cite{Smart99}) 
which improves the practicality of solving the $S$--unit equation (\ref{SunitEquation}). 

Using the above algorithmic methods, how to solve the $S$--unit equation (\ref{SunitEquation}), 
unable us to find all $\lambda$. Thus, the new algorithm we suggest has three main parts which are 
the following, given $K$ and $S$ as input:

\begin{enumerate}
\item[$\bullet$] Find the finite set of all possible 2--division fields $L/K$ where $\lambda$ may 
	lie.
\item[$\bullet$] For each extension $L/K$ solve the related $S$--unit equation $\lambda+\mu=1$ only 
	for $\lambda$ and $\mu$ such that the associated $j\in K$.
\item[$\bullet$] For each $\lambda$ we evaluate $j$ and then we find the elliptic curves.
\end{enumerate}
Even though our algorithm is based on the above general algorithmic methods for solving $S$--unit 
equations in $L$, we will be able to derive additional conditions on $\lambda$. The main such 
additional condition is simply that we may also assume that the value of $j$ (obtained from $\lambda$ 
via (\ref{j_lambda})) lies in $K$; this restricts $\lambda$ to lie in a subgroup of $\SuL$ of 
substantially smaller rank.

This paper has five sections. In the first we fix some of the notation and describe the basic 
properties of the 2--division field. In the second section we describe how we can compute the set of 
all possible 2--division fields $L/K$ using Kummer theory and we prove that there are finitely many 
of these. In the third section we give necessary and sufficient conditions of $\lambda$ to be at the 
same time a solution of the $S$--unit equation (\ref{SunitEquation}) and the $\lambda$--invariant of 
an elliptic curve defined over $K$. These conditions make the algorithmic methods of solving equation 
(\ref{SunitEquation}) more effective by allowing us to replace the finitely generated group $\SuL$ 
with one subgroup of smaller rank. We also explain how we can reduce the number of $S$--unit 
equations we have to solve. In the fourth section we show how we can modify Smart and Windanger's 
ideas in order to speed up the sieve step in the solution of (\ref{SunitEquation}). In the final 
section we give examples of elliptic curves in the case when $K=\QQ$ or a quadratic field, and we 
compare our method with current implementations.

Unfortunately, as far as we know there is no implementation in any known software package(Sage, 
Magma etc) that solves general $S$--unit equations. Since our new method is based on the algorithmic 
solution of $S$--unit equations we had to implement such an algorithm. The importance of an effective 
implementation of solving general $S$--unit equations is not limited to our problem, but may be 
applied to many other problems in Number Theory and Diophantine equations. We stress that it is not 
trivial at all to implement such an algorithm. 

%There are many difficulties and technical points that do not exist in the theoretical approach.

Acknowledgements: I would to thank my supervisor professor John Cremona for suggesting me this 
problem, for the long discussions we had and his constant encouragement. Also, for the suggestion of 
Proposition (\ref{PropHilbertSymbol}) and his comments on an earlier draft of this article. Moreover, 
I want to thank ICERM and Brown University because a big part of the paper was written during my 
visit in the institute. Finally, I would like to thank Bekiari-Vekri foundation and the Academy of 
Athens for funding a part of my studies and EPSRC for covering a part of my fees.

\section{Notation--Basic definitions}

Let $K$ be a number field and $S$ a finite set of primes of $K$. Let $E$ be an elliptic curve over 
$K$ in the long Weierstrass form 
$$E:y^2+a_1xy+a_3=x^3+a_2x^2+a_4x+a_6$$
which has good reduction outside the set $S$. Let $j$ be the $j$--invariant of the curve and $\Delta$ 
its discriminant.

We define the ring of $S$--integers $\SiK$ of $K$, the $S$--unit group $\SuK$ of $K^*$ and the $n$--Selmer group of $K$ and $S$ to be
\begin{align*}
\SiK&=\lbrace x\in K^*\mid\ord_\pK(x)\geq0,\forall\pK\not\in S\rbrace\\
\SuK&=\lbrace x\in K^*\mid\ord_\pK(x)=0,\forall\pK\not\in S\rbrace.\\
K(S,n)&=\left\lbrace x\in K^*/K^{*n}\mid\ord_\pK (x)\equiv 0\mod n,\forall\pK\not\in S 
\right\rbrace
\end{align*}
It is known that $\SuK$ is a finitely generated abelian group and $K(S,n)$ a finite abelian
group(see \cite{CohenAdvanced}). For $n\mid m$ we define $K(S,n)_m$ to be the image of the natural 
map $K(S,m)\rightarrow K(S,n)$. For a positive natural number $n$ we also define
$$S^{(n)}=S\cup\left\lbrace\pK\mid\ord_\pK(n)>0\right\rbrace.$$

In order to determine all elliptic curves $E/K$ with good reduction outside $S$ it is enough to 
determine the $j$--invariants of the isomorphic classes of the curves(\cite{CremonaLingham07}). 
The cases $j=0$ and $j=1728$ were treated fully in \cite{CremonaLingham07} and so we omit these from consideration throughout this paper. For other $j$ we have the following,

\begin{proposition}[\cite{CremonaLingham07}]\label{PropCremonaLingham}
Let $E$ be an elliptic curve defined over $K$ with good reduction at all primes $\pK\not\in S$ and 
$j\neq 0, 1728$. Set $w:=j^2(j-1728)^3$. Then 
\begin{align*}
\Delta&\in K(S,12),&j&\in\SiK,&w&\in K(S,6)_{12}.
\end{align*}
Conversely, if $j\in\SiK$ with $w\in K(S,6)_{12}$ then there exists elliptic curve $E$ with $j(E)=j$ 
and good reduction outside $S^{(6)}$.
\end{proposition} 

An explicit equation of the curve $E$ whose existence is claimed in the above proposition is the 
following,
$$E:y^2=x^3-3u^2j(j-1728)x-2u^3j(j-1728)^2$$
with $u\in K^*$ is such that $(3u)^6w\in K(S,12)$. It is shown in \cite{CremonaLingham07} that this 
curve has good reduction outside $S^{(6)}$. The other curves with good reduction outside $S$ and the 
same $j$--invariant are all twists $E^{(u)}$ of $E$ by $u\in K(S,2)$. 

Recall that the 2--division polynomial of $E$ is defined by,
$$f_2 = 4x^3+b_2x^2+2b_4x+b_6$$
where $b_2=a_1^2+4a_2$, $b_4=2a_4+a_1a_3$ and $b_6=a_3^2+4a_6$. Denote the roots of $f_2$ by 
$e_1,e_2,e_3$ then the 2--division field of $E$ is $L=K(e_1,e_2,e_3)$.

\begin{definition}
The $\lambda$--invariant of $E$ is,
$$\lambda=\frac{e_1-e_3}{e_1-e_2}.$$
\end{definition}
As it stands $\lambda$ is not well--defined since the numbering of the $e_i$ is not fixed. We use 
$\lambda$ to denote any of the six values of the set,
$$\Lambda:=\left\lbrace\lambda,1-\lambda,\frac{1}{\lambda},\frac{1}{1-\lambda},1-\frac{1}{\lambda},
\frac{\lambda}{\lambda-1}\right\rbrace$$
We define $\mu:=1-\lambda$. As we already mentioned in the introduction we have,
$$j = 2^8\frac{(\lambda (\lambda-1)+1)^3}{\lambda^2(1-\lambda)^2}.$$
An easy calculation shows that any element of the set $\Lambda$ gives the same $j$.

Since, $L$ is the splitting field of $f_2$ then $L/K$ is a Galois extension and its Galois group 
$\Gal(L/K)$ is isomorphic to a subgroup of $S_3$. Moreover, since $\Delta (2^4f_2) = k^2\Delta(L/K)$ 
for some $k\in\OK$ and 
\begin{equation}\label{discriminants_relation}
2^4\Delta=\Delta (f_2)=2^8\prod_{1\leq i<j\leq 3}(e_i-e_j)^2
\end{equation}
we see that $\pK\nmid 2\Delta\Rightarrow\pK\nmid \Delta(L/K)$. Also, we observe that $L=K(\lambda)$. 
We summarize the situation in the following proposition,

\begin{proposition}\label{PropositionDivisionFieldProperties}
The $2$--division field $L$ of $E$ is a Galois extension over $K$ unramified at all primes not in 
$S^{(2)}$, and has Galois group isomorphic to a subgroup of $S_3$.
\end{proposition}

If $S_L:=\left\lbrace\bL\subset\mathcal O_L:\bL\mid\pK,\text{ for some }\pK\in S^{(2)}\right\rbrace$, 
then by the above equation (\ref{discriminants_relation}) we deduce that all $e_i-e_j$ are not 
divisible by primes not in $S_L$, and so they are $S_L$--units. As a result, $\lambda$ and $\mu$ are 
both $S_L$--units and are solutions of the $S$--unit equation (\ref{SunitEquation}) over the 
$2$--division field $L$ for the set of primes $S_L$.

For an extension $N/M$ and a set $S_N$ of prime ideals of $N$ we define
\begin{align*}
\mathcal O^*_{N,M,S_N,1}&:=\left\lbrace x\in\mathcal O^*_{N,S_N}|\Norm_{N/M}(x)=1\right\rbrace\\
\mathcal O^*_{N,M,S_N,\pm 1}&:=\left\lbrace x\in\mathcal O^*_{N,S_N}|\Norm_{N/M}(x)=\pm 1 
\right\rbrace
\end{align*}

Finally, for a place $\pK$ we define its absolute value to be,
$$|x|_\pK=\begin{cases}
p^{-f_\pK\ord_\pK(x)}, & \text{if }\pK\text{ is a finite prime.}\\
|\sigma_\pK(x)|, & \text{if }\pK\text{ is a real place.}\\
|\sigma_\pK(x)|^2, & \text{if }\pK\text{ is a complex place.}
\end{cases}$$
where for $\pK$ an infinite place, $\sigma_\pK$ denotes the associated embedding into $\mathbb R$ or 
$\mathbb C$, and $f_\pK$ its residual degree when $\pK$ is an finite prime.

\section{Computation of 2--Division Fields}\label{DivisionFields}

Let $K$ be a number field and $S$ a finite set of primes $K$. We are looking for Galois extensions of 
$K$ with $\Gal(L/K)$ isomorphic to a subgroup of $S_3$. Since $S_3$ is a solvable group we can 
construct $L/K$ as a tower of cyclic extensions. We use Kummer theory to construct cyclic extensions 
even though there are other methods\footnote{Class Field theory can also be used for constructing 
cyclic extensions with the desired properties. See \cite{CohenAdvanced} for a detailed description.}. 
A detailed description of Kummer theory with a computational point of view and all proofs of the 
theorems we use can be found in \cite{CohenAdvanced} and \cite{CohenCourse}.

%Let $n>1$ be a natural number and $\zeta_n$ a primitive the n--th root of unity. Then the main 
%theorem of Kummer theory is the following,
%
%\begin{theorem}[Kummer]\label{KummerMainTheorem} 
%Let $n>1$ be an integer and $K$ be a number field such that $\zeta_{n}\in K$. If $K_n$ is the maximal 
%abelian extension of exponent $n$ then,
%$$\Hom_{cts}(\Gal(K_n/K),K)\simeq H^1(G_K,K)\simeq K^*/K^{*n},$$
%where $G_K$ is the absolute Galois group of $K$.
%\end{theorem} 

%From the main theorem of Kummer theory and the fact that an arbitrary abelian extension of $K$ is a 
%tower of prime extensions, we focus only in the case when $n$ is equal to a prime number $p$. For the 
%case when $K$ contains a suitable root of unity the following proposition holds.

Since a cyclic extension is a tower of prime extensions, we focus only on prime degree extensions. By 
Kummer theory we have,

\begin{proposition}\label{PropositionUnramifiedC2} 
Let $K$ be a number field, $p$ a natural prime number such that $\zeta_{p}\in K$ and $S$ a finite set 
of primes $K$. Let $L=K(\sqrt[p]{a})$ with $a\in K/K^{*p}$; if $L/K$ is unramified at all the primes 
$\pK\not\in S$, then $a\in K(S,p)$.
\end{proposition}

In case $p$ is an odd prime and $\zeta_p\not\in K$ we have to work with the extension 
$K_z=K(\zeta_p)$. We first focus on $p=3$ because we are able to give a simple defining polynomial of 
the extension. We define $S_{K_z}=\left\lbrace\mathfrak B\subseteq \mathcal O_{K_z}|\exists 
\mathfrak p\in S\text{ s.t.}\mathfrak {B|p}\right\rbrace$. Then we have:

\begin{proposition}\label{PropositionUnramifiedC3}
Let $K$ be a number, $\zeta_3\not\in K$ and $S$ a finite set of prime ideals of $\OK$. Let $L/K$ be a 
Galois 3 degree extension which is unramified at all primes outside $S$. Then there exists $a\in 
K_z(S_{K_z},3)$ such that $N_{K_z/K}(a)=\beta^3$ for $\beta\in K$ and $L=K(\theta)$ with $\theta$ a 
root of the polynomial $f(x)=x^3-3\beta x-Tr_{K_z/K}(a)$.
\end{proposition}

By Propositions \ref{PropositionUnramifiedC2} and \ref{PropositionUnramifiedC3} we know how to 
construct all the Galois $C_2$ and $C_3$ extensions of a number field $K$ which are unramified 
outside a finite set $S$ of prime ideals of $K$. For the $S_3$ case we use the fact that $S_3$ is a 
solvable group and a tower of a quadratic and cubic extensions. Even though we care about $S_3$ 
extensions in the rest of the section we describe an algorithm of constructing dihedral extensions 
$D_p$ with $p$ an odd prime $p\equiv 3\mod 4$.

Let $K\subset M\subset L$ be a tower of Galois extensions where $\Gal(M/K)\simeq C_2$ and $\Gal(L/M) 
\simeq C_p$ with $p$ be an odd prime $p\equiv 3\mod4$. The details of computing general $C_p$ 
extensions can be found in \cite{CohenAdvanced}. We fix the notation such that $\Gal(L/M)= \langle 
\sigma\rangle$, $\Gal (M/K)=\langle\tau\rangle$, $f(x)\in M[x]$ is a defining polynomial of 
$L/M$ with the coefficient of $x^{p-1}$ to be 0, $\ft =\tau (f)$ and $\Lt$ is the splitting field of 
$\ft$. We use the same notation for lifts of $\tau$ and $\sigma$ to the absolute Galois group $\Gal 
(\bar K/K)$.

\begin{proposition}
We have that $\Lt/M$ is a Galois $C_p$ extension.
\end{proposition}

\begin{proof}
Clear.
\end{proof}

\begin{theorem}\label{TheoremIsGalois}
The extension $L/K$ is Galois if and only if $L=\Lt$.
\end{theorem}

\begin{proof}
Clear.
%Keep it
%Let $\bar K$ be the algebraic closure of $K$. It holds that $L/K$ is Galois $\Leftrightarrow\forall 
%\phi\in\Gal(\bar K/K)\Rightarrow\phi(L)=L\Leftrightarrow\sigma(L)=L$ and $\tau(L)=L\Leftrightarrow 
%\tau(L)=L$ since $L/M$ is Galois. 
\end{proof}

%Now assuming that $L/K$ is Galois. We need the following classical result,
%
%\begin{proposition}\label{Proposition_S3_or_C3} 
%Let $K$ be a number field and $f(x)\in K[x]$ an irreducible polynomial with $deg(f)=n$. If $L$ is the 
%splitting field of $f(x)$ then we know that $\Gal(L/K)\subset S_n$. If $A_n$ is the alternative group 
%then $\Gal(L/K)\subset A_n$ if and only if $disc(f)$ is a square.
%\end{proposition}

Now we assume that $L/K$ is Galois. The group $\Gal(L/K)$ acts on the roots of $f$ as a subgroup of 
$S_p$. Since $p\equiv3\mod4$ we know that $C_p$ is a subgroup of the alternative group $A_p$ but 
$D_p$ is not. The goal is to find an irreducible polynomial $h$ of degree $p$ over $K$ with splitting 
field $K(h)$ a subfield of $L$. Considering discriminants, we are able to check the case 
$\Gal(K(h)/K)=C_p$ and as a result to distinguish\footnote{$D_p$ does not have a normal subgroup of 
order 2 while $C_{2p}$ has one.} $\Gal(L/K)=D_p$ and $C_{2p}$.

For the special case $f=\ft$(i.e. $f\in K[x]$) we have:

\begin{lemma}\label{LemmafoverK}
Let $f$ be a defining polynomial of $L/M$ with $f\in K[x]$. Then $\Delta (f)\in K^*\setminus 
K^{*2}$ if and only if $\Gal(L/K)\simeq D_p$.
\end{lemma}

\begin{proof}
Let $K(f)$ be the splitting field of $f$ over $K$. Since $L/K$ is Galois and $f$ a defining poynomial 
of $L/M$ we have $K\subset K(f) \subset L$. Then $\Delta (f)\in K^{*2}\Leftrightarrow\Gal(K(f)/K)=C_p 
\Leftrightarrow\Gal(L/K)=C_{2p}$.
\end{proof}

Now we assume that $f\neq\ft$. We use the roots of $f$ and $f^\tau$ to define a polynomial $h\in 
K[x]$ of degree $p$ whose splitting field is $L$. Write $f(x)=(x-a_1)(x-a_2)\cdots(x-a_p)$ and 
$\ft(x)=(x-b_1)(x-b_2)\cdots(x-b_p)$ with $a_i,b_j\in L$. 

In case $\Gal(L/K)=D_p$ we may assume that $\sigma$ permutes the roots $a_i$ and $b_j$ as $\sigma 
=(a_1 a_2\cdots a_p)(b_1 b_2\cdots b_p)$. Also we may assume that $\tau (a_1)=b_1$ and 
$\tau\sigma\tau =\sigma^{-1}$. Then by induction we get that $\tau (a_i)=b_{p+2-i\mod p}$ for $i=1,
\cdots ,p$. We define $h$ to be,
\begin{align}\label{definitionh}
h(x)&=(x-a_1-b_1)(x-a_2-b_2)\cdots(x-a_p-b_p)\\
&=(x-a_1-\sigma (a_1))(x-a_2-\sigma (a_2))\cdots(x-a_p-\sigma(a_p))\nonumber
\end{align}
%$$h(x)=(x-(a_1+b_1))(x-(a_2+b_2))\cdots(x-(a_p+b_p))$$
with\footnote{If $a_1+b_1=0$ then $a_1+b_2\neq 0$ since $b_1\neq b_2$. So we can change $h(x)$ with 
$h^\prime (x)=(x-a_1-\sigma (b_1))(x-a_2-\sigma (b_2))\cdots (x-a_p-\sigma (b_p))$.} $a_1+b_1\neq 0$. 
We can easily check that $\tau (h)=h$ and $\sigma (h)=h$, so $h\in K[x]$. We recall that we have 
assumed that the coefficient of $x^{p-1}$ of $f$ is equal to 0. 

\begin{lemma}
If $h$ and $f$ are as above then $h$ is irreducible in $K[x]$ and $K\subset K(h)\subset L$. 
\end{lemma}

\begin{proof}
The fact that $K\subset K(h)\subset L$ comes from the definition of $h$ and the fact that $L/K$ is 
Galois.

Let assume that $h$ is not irreducible in $K[x]$. Since $K\subset K(h)\subset L$, $[L:K]=2p$ and the 
degree of $h$ is $p$ we conclude that $h$ has a root in $K$. Without loss of generality we assume 
that $a_1+b_1=\delta\in K$. Applying $\sigma$ to $\delta$ we have $a_1+b_1=a_2+b_2=\cdots a_p+b_p= 
\delta$. Since the coefficient of $x^{p-1}$ in $f$ is equal to 0 we have that $a_1+a_2\cdots+a_p=b_1 
+b_2+\cdots +b_p=0$ which means $(a_1+b_1)+(a_2+b_2)\cdots+(a_p+b_p)=0\Rightarrow p\delta =0$ and we 
get $\delta =0$. So, we have $a_1+b_1=0$, contradiction.
\end{proof}

We summarize the previous results in the following theorem,

\begin{theorem}\label{TheoremS_3Distinction}
Let $L/K$ be a Galois extension of degree $2p$ for $p\equiv 3\mod4$ and $M$ its quadratic subfield. Let $\tau\in\Gal(L/K)$ be an element of order 2, $f\in M[x]$ a defining polynomial of $L/M$ such that the coefficient of $x^{p-1}$ is zero, $\ft=\tau(f)$ and $h$ as in (\ref{definitionh}). Then we have,
\begin{enumerate}[(i)]
\item If $f=f^\tau$ then $\Delta (f)\in K^*\setminus K^{*2}$ if and only if $\Gal(L/K)\simeq D_p$.
\item If $f\neq f^\tau$ then $h$ is irreducible and $h\in K[x]$. Moreover, $\Delta (h)\in 
K^*\setminus K^{*2}$ if and only if $\Gal(L/K)\simeq D_p$.
\end{enumerate}
\end{theorem}
  
Since the $p$--Selmer group of any number field with respect to any finite set of primes $S$ is a 
finite abelian group, we can deduce now the following,

\begin{theorem}\label{TheoremFiniteExtensions}
Let $K$ be a number field, $S$ a finite set of prime of $K$. Then the number of Galois extensions 
$L/K$ unramified outside $S$ with Galois group equal to $C_2$, $C_3$ or $S_3$ is finite.
\end{theorem} 

Now we have all the ingredients to construct all possible 2--division fields of the elliptic curves 
we are looking for. Propositions \ref{PropositionUnramifiedC2} and \ref{PropositionUnramifiedC3}
explain how to get the 2--division fields with Galois group $C_2$ or $C_3$. Using the same 
propositions we construct all extensions $L/K$ with degree $6$ and unramified outside $S$. By 
Theorems \ref{TheoremIsGalois} and \ref{TheoremS_3Distinction} we can understand 
which of these extensions are Galois with Galois group isomorphic to $S_3$.

\section{Solving $S$--unit equations}\label{UnitEquation}

We recall that $K$ is a number field, $S$ a finite set of prime ideals of $K$, $E$ is an elliptic 
curve over $K$ with good reduction outside $S$, $\lambda$ is the $\lambda$--invariant of $E$ and 
$\mu=1-\lambda$, $L$ is the $2$--division field of $E$, $S_L=\left\lbrace\bL\subset\mathcal O_L:\bL 
\mid\pK,\text{ for some }\pK\in S^{(2)}\right\rbrace$ and $\SuL$ the $S$--unit group of $L$ with 
respect to $S_L$.

In this section we show that $\lambda$ and $\mu$ lie in a smaller subgroup than the full $\SuL$. That 
speeds up both the second and third steps of the algorithm. We will not present the algorithmic 
methods of solving a general $S$--unit equation since this has been described fully elsewhere as we 
have already mentioned(\cite{DeWegerthesis}, \cite{DeWeger87}, \cite{TzanakisDeWeger89}, 
\cite{TzanakisDeWeger91}, \cite{TzanakisDeWeger92},\cite{Smart95}, \cite{Smart99},\cite{Wildanger00}, 
\cite{Smartbook}). 

The group $\pgl$ acts on $K$ with the usual way, 
$$\left(\begin{matrix}
a & b\\
c & d
\end{matrix}\right)\cdot z=\frac{az+b}{cz+d}$$
for $z\in K$ and $\left(\begin{matrix}
a & b\\
c & d
\end{matrix}\right)\in\pgl$.

Since we have assumed that $j\neq 0, 1728$, by equation (\ref{j_lambda}) we see that all the elements 
of $\Lambda$ are distinct. We define 
$$T=\left(\begin{matrix} 
0 & 1\\
1 & 0\\
\end{matrix}\right), 
R=\left(\begin{matrix}
0 & 1\\
-1 & 1\\
\end{matrix}\right)\in\pgl$$
where $T$ has order 2 and $R$ has order 3. If $G:=\langle T,R\rangle\simeq S_3$ then $G$ acts on 
$\Lambda$. Actually, $G$ acts as a permutation group on the roots of the polynomial\footnote{$F_j(x)$ 
may not be irreducible. The crucial thing is that $F_j(x)$ has coefficients in $K$.} $F_j(x)= 
\displaystyle\prod_{\lambda^\prime\in\Lambda}(x-\lambda^\prime)=(x^2-x+1)^3-\frac{j}{2^8} x^2
(1-x)^2\in K[x]$. Because $L=K(\lambda)$ the splitting field of $F_j(x)$ is $L$. From the above we 
can see that $\Gal(L/K)$ acts on the set $\Lambda$ in the same way as a subgroup of $G$. 

We divide into cases, according to the structure of $\Gal(L/K)$.

\subsection{$\Gal(L/K)$ is trivial}

When $\Gal(L/K)$ is trivial the elliptic curve has full 2--torsion and $\lambda,\mu\in K$. We have to 
solve (\ref{SunitEquation}) for $\lambda,\mu\in\mathcal O^*_{K,S^{(2)}}$. 

\subsection{$\Gal(L/K)\simeq C_2$}

In case $\Gal(L/K)\simeq C_2$, the following theorem holds,

\begin{theorem}\label{TheoremC2}
Suppose that $\Lambda\subset\SuL$ is associated with an elliptic curve defined over $K$. Then there 
exists $\lambda\in\Lambda$ such that $\lambda\in\SuLKo$.

Conversely, let $\lambda\in\SuL$ and $j,w,\mu$ as above. If $\mu\in\SuL$ and $\lambda\in\SuLKo$ then 
$j\in\SiK$. Moreover, if $w\in K(S,6)_{12}$ then $j$ is the $j$--invariant of an elliptic with good 
reduction outside $S^{(6)}$.
\end{theorem}

\begin{proof}

If $\tau$ is a generator of $\Gal(L/K)$ then we can choose $\lambda\in\Lambda$ from the compatibility 
of the action of $G$ and $\Gal(L/K)$ on the set $\Lambda$ such that $\tau(\lambda)=T\cdot\lambda= 
\frac{1}{\lambda}$. As a result, $$N_{L/K}(\lambda)=1.$$

Conversely, if $N_{L/K}(\lambda)=1$ that means $\tau(\lambda)=\frac{1}{\lambda}$ and then 
$\tau(j)=j$. Since $\lambda,\mu\in\SuL$ we have that $\ord_\bL(\lambda^2-\lambda+1)\geq 0$ for all 
$\bL\not\in S_L$ and then $j\in\SiK$. If $w\in K(S_K,6)_{12}$ then by Proposition 
\ref{PropCremonaLingham} there exists an elliptic curve over $K$ with $j$--invariant equal to $j$ and 
good reduction at all prime not in $S^{(6)}$.
\end{proof}

\subsection{$\Gal(L/K)\simeq C_3$}

In case $\Gal(L/K)\simeq C_3$, the following theorem holds,

\begin{theorem}\label{TheoremC3}
Let $\Gal(L/K)=\langle\sigma\rangle$. Suppose that $\Lambda\subset\SuL$ is associated with an 
elliptic curve defined over $K$. Then there exists $\lambda\in\Lambda$ satisfying the following conditions,
\begin{enumerate}[(i)]
\item $-\lambda\in\SuLKo$.
\item $\sigma(\lambda)=\frac{1}{\mu}$.
\end{enumerate}

Conversely, let $\lambda\in\SuL$ and $j,w,\mu$ as above. If $\mu\in\SuL$ and (i)--(ii) hold then 
$j\in\SiK$. Moreover, if $w\in K(S,6)_{12}$ then $j$ is the $j$--invariant of an elliptic with good 
reduction outside $S^{(6)}$.
\end{theorem}

\begin{proof}
As in the quadratic case we can assume that $\sigma(\lambda)=S\cdot\lambda=\frac{1}{1-\lambda} 
=\frac{1}{\mu}$. Then an easy calculation shows that $N_{L/K}(-\lambda)=1$. The proof of the converse 
is similar to Theorem (\ref{TheoremC2}).
\end{proof}

From the condition $\sigma(\lambda)=\frac{1}{\mu}$ we also see that $N_{L/K}(-\mu)=1$. Actually, one 
can prove that we can replace condition (ii) with the condition $-\mu\in\SuLKo$. As a result, we have 
that $\lambda,\mu\in\SuLKpmo$.

\subsection{$\Gal(L/K)\simeq S_3$}
Finally, for the case $\Gal(L/K)\simeq S_3$ we have similar result,

\begin{theorem}\label{TheoremS3}
Let $\Gal(L/K)=S_3=\langle\sigma,\tau\rangle$ such that $\sigma^3=\tau^2=1$. Suppose that $\Lambda \subset\SuL$ is associated with an elliptic curve defined over $K$. Then there exists $\lambda\in \Lambda$ satisfying the following conditions,
\begin{enumerate}[(i)]
\item $\lambda\in\mathcal O^*_{L,L^\tau,S_L,1}$.
\item $-\lambda\in O^*_{L,L^\sigma,S_L,1}$
\item $\sigma(\lambda)=\frac{1}{\mu}$.
\end{enumerate}

Conversely, let $\lambda\in\SuL$ and $j,w,\mu$ as above. If $\mu\in\SuL$ and (i)--(iii) hold then $j\in\SiK$. Moreover, if $w\in K(S,6)_{12}$ then $j$ is the $j$--invariant of an elliptic with good reduction outside $S^{(6)}$.
\end{theorem}

\begin{proof}
As in the proofs of Theorems (\ref{TheoremC2}) and (\ref{TheoremC3}) we can assume that 
$\sigma(\lambda)=S\cdot\lambda=\frac{1}{1-\lambda}=\frac{1}{\mu}$ and $\tau(\lambda)=T\cdot\lambda=
\frac{1}{\lambda}$. The last two equations show that $N_{L/L^{\sigma}}(-\lambda)=1$ and 
$N_{L/L^\tau}(\lambda)=1$, respectively. The proof of the converse is similar to theorem 
(\ref{TheoremC2}).
\end{proof}

Since, we have proved that $\lambda$ and $\mu$ are constrained to lie in different subgroups of 
$\SuL$ we denote by $\Gl$ and $\Gm$ the groups where $\lambda$ and $\mu$ lie. If $M_L$ is the set of 
places of $L$ then we define,
\begin{align*}
\Sl&:=\left\lbrace \bL\in M_L:|x|_\bL\neq 1\text{ for some } x\in\Gl\right\rbrace\\
\Sm&:=\left\lbrace \bL\in M_L:|x|_\bL\neq 1\text{ for some } x\in\Gm\right\rbrace
\end{align*}

By the relation $\sigma(\lambda)=\frac{1}{\mu}$ in the $C_3$ and $S_3$ cases, we can make a choice of 
bases of $\Gl$ and $\Gm$ such that $\lambda$ and $\mu$ have the same vector of exponents.

Since, solving $S$--unit equations it is the hardest part of the method, we want to avoid doing more 
computations than necessary. To complete this section we include several results which in practice 
reduce the number of $S$--unit equations which need to be solved.

In the case where the initial set of primes $S$ does not include a prime $\pK$ above 2 the following 
two propositions hold,

\begin{proposition}\label{PropNotPrimeAbove2C2C3}
If $[L:K]=2,3$ and $E$ has good reduction at a prime $\pK\mid 2$ then $\ord_\pK(\Delta(L/K))\equiv 0 
\mod 2$.
\end{proposition}

\begin{proof}
There exists $k\in K^*$ such that $k^2\Delta(L/K)=\Delta(2^4f_2) =2^{16}\Delta(f_2)$. By Lemma 3.1 in 
\cite{CremonaLingham07} we know that $\ord_\pK(\Delta)\equiv 0\mod 12$. Finally, using equation 
(\ref{discriminants_relation}) we get the result.
\end{proof}

\begin{proposition}\label{PropNotPrimeAbove2S3}
If $[L:K]=6$, $L_c$ a cubic subfield and $E$ has good reduction at a prime $\pK\mid 2$ then 
$\ord_\pK (\Delta(L_c/K))\equiv 0\mod 2$.
\end{proposition}

\begin{proof}
We know that there exists $k\in K^*$ such that $k^2\Delta(L_c/K)=\Delta(2^4f_2)=2^{16}\Delta(f_2)$. 
The rest is as above.
\end{proof}

Because finding isogenous curves is a quite fast procedure, in practice we also assume that we are 
looking for curves up to isogeny in order to reduce the number of $S$--unit equations we have to 
solve. When $E$ has full two torsion we have,

\begin{proposition}\label{PropRibetMazur}
If $E$ has full two torsion then it is isogenous to one without full two torsion.
\end{proposition}

\begin{proof}
See Proposition 2.1 of \cite{Ribet76}.
\end{proof}

Proposition \ref{PropRibetMazur} actually says that we do not have to solve $S$--equations in the 
case $L=K$.

Let $E/K$ be an elliptic curve with a rational 2--torsion point. We assume that $E$ is of the 
following form,
$$E:y^2=x(x^2+ax+b)\text{, with }a,b\in\OK$$
Then, the 2--isogenous of $E$ is the curve 
$$\bar E:Y^2=X(X^2+\bar aX+\bar b)$$
where $\bar a=-2a$ and $\bar b=a^2-4b$. An easy calculation shows that $\Delta(E)=2^4b^2(a^2-4b)$ and 
$\Delta(\bar E)=2^4\bar b^2(\bar a^2-4\bar b)=2^8b(a^2-4b)^2$.

\begin{proposition}\label{PropHilbertSymbol}
Let $E,\bar E$ be as above. If $\left(\frac{\cdot,\cdot}{K}\right)$ is the Hilbert  symbol relative 
to $K$ then we have,
$$\left(\frac{\Delta(E),\Delta(\bar E)}{K}\right)=1.$$
\end{proposition}

\begin{proof}
The equation $\Delta(E)x^2+\Delta(\bar E)y^2=z^2$ has the non-trivial solution $(1,2,a)$ in $\OK^3$.
\end{proof}

%Since, we can calculate the $j$--invariant of $E$ knowing the $j$--invariant of $\bar E$, 

The previous two propositions help us to reduce the number of $S$--unit equations we have to solve. 
If $L$ is the 2--division field of an elliptic curve with only one rational point of order 2 then 
$\Delta(E)\equiv d\mod K^{*2}$ where $L=K(\sqrt{d})$ for $d\in K(S_K,2)$ by Proposition 
\ref{PropositionUnramifiedC2}. We find the smallest set $D$ of $d$'s such that at least one of the 
$d_1,d_2$ belongs in $D$ when $\left(\frac{d_1,d_2}{K}\right)=1$ for all possible 
combinations of $d_1,d_2$(including the case $d_1=d_2$). Then we have to solve the $S$--unit equation 
(\ref{SunitEquation}) only for the cases when $d\in D$ and then find all the isogenous curves.

\section{Efficient Sieve}

As we briefly explained in the introduction, the generalized $S$--unit algorithm to find all 
solutions of equation (\ref{SunitEquation}) where $\lambda,\mu$ lie in (possibly different) finitely 
generated subgroups of $L^*$ for some number field $L$ with generators $\lambda_0,\lambda_1,\cdots,
\lambda_n$ and $\mu_0,\mu_1,\cdots,\mu_k$ has three main steps:
\begin{enumerate}[(1)]
\item Use bounds of linear forms and $p$--adic logarithms to obtain bounds for the exponent vectors 
of $\lambda,\mu$ for any solutions.

\item Use lattice basis reduction algorithms and LLL--reduction to reduce these bounds as much as 
possible.

\item Use a sieve method to discard many candidate $\lambda,\mu$ which do not give solutions.
\end{enumerate}

General sieve methods have been suggested by Smart and Wildanger(\cite{Smart99}, \cite{Wildanger00}). 
However, we benefit from Theorems \ref{TheoremC2}, \ref{TheoremC3} and \ref{TheoremS3} and the 
symmetries they introduce. We modify Smart and Wildanger's ideas avoiding the use of Finke--Pohst 
algorithm (\cite{FinckePohst85}) in any point of the sieve as they suggest. By contrast with Smart 
and Wildanger we allow each generator to have its own upper bound for the absolute value of its 
exponent. Also it is important to mention that a choice of basis for $\Gl$ and $\Gm$ such that some 
of the generators are units(if it is possible) is crucial.

For the rest of this section we fix bases for $\Gl$ and $\Gm$ and we express $\lambda$ and $\mu$ as 
a multiplicative combination of the bases,
\begin{align*}
\lambda&=\prod_{i=0}^n\lambda_i^{x_i} & \mu&=\prod_{i=0}^m\mu_i^{y_i}
\end{align*}
We always assume that $\lambda_0$ and $\mu_0$ are the generators of the torsion part of $\Gl$ and 
$\Gm$, respectively. We also assume that from steps (1) and (2) of the algorithm we have two vectors 
$B_0=(b_0^0,b_1^0,\cdots,b_n^0)$ and $C_0=(c_0^0,c_1^0,\cdots,c_m^0)$ such that for every solution, 
$|x_i|\leq b_i$ and $|y_j|\leq c_j$ for all $i=0,\cdots,n$ and $j=0,\dots, m$. 

For a fixed subset $I$ of $\{0,1,\cdots,n\}$ we define 
\begin{align*}
S_I&=\left\lbrace\bL\in\Sl:\exists\lambda_i \text{ with }i\in I\text{ such that }|\lambda_i|_\bL\neq 
1 \right\rbrace.\\
\lI&=\prod_{i\in I}\lambda_i^{x_i}.
\end{align*}
We denote by $I^\infty=\{i\in\{0,1,\cdots,n\}:\lambda_i\text{ is a unit}\}$ and $\linf=\lambda_{I^ 
\infty}$. Similarly, we define $S_J$ for $J$ a fixed subset of $\{0,1,\cdots,m\}$, $J^\infty$ and 
$\minf$. Since we use Theorems \ref{TheoremC2}, \ref{TheoremC3} and \ref{TheoremS3}, the sieve 
depends on $\Gal(L/K)$. 

\paragraph*{$C_2$ case}

Let $\tau$ be the generator of $\Gal(L/K)$. By the first part of Theorem \ref{TheoremC2} we 
understand that $\Gm=\SuL$, and $\Sl$ contains only split primes $\bL$ such that 
$\ord_\bL(g)=-\ord_{\tau(\bL)} (g)$ for all $g\in\Gl$. For each one of the conjugate finite primes 
$\bL\in\Sl$ we prove that there are no pairs of solutions $(\lambda,\mu)$ for which $|\ord_\bL 
(\lambda)|$ is `large'. We do this by showing that,
$$|\mu-1|_\bL <\delta\ll 1$$
has no non--trivial solutions, using Lemma 4 in \cite{Smart99}. We use the new upper bounds on 
$|\ord_\bL(\lambda)|$ obtained in this way to get new vectors $B_1=(b_0^1,b_1^1,\cdots,b_n^1)$ and 
$C_1=(c_0^1,c_1^1,\cdots,c_m^1)$. 

Since we have chosen only finite primes $\bL\in\Sl$ we have not reduced the exponents bounds $b_i^1$ 
and $c_j^1$ for the unit generators. The way we reduce the bounds for the unit generators is to split 
the set of solutions in two sets, where the first set contains solutions with smaller exponents, and 
try to show that the second set contains no solutions. In order to do that we need a few definitions. 
Let $\Slinf:=S_{I^\infty}$ then\footnote{Note that $\bL\in\Slinf\Leftrightarrow\tau(\bL)\in\Slinf$.}, 

\begin{definition}
Let  $B=(b_0,b_1,\cdots,b_n)\in\NN^{n+1}$ and $C=(c_0,c_1,\cdots,c_n)\in\NN^{m+1}$. Then for $R>1$ we 
define,
$$\mathcal L^2_\infty(B,C,R)=\left\lbrace(\lambda,\mu):|x_i|\leq b_i,|y_i|\leq c_i\text{ and }|\log|
\lambda|_\bL|\leq\log(R),\forall\bL\in\Slinf\right\rbrace.$$
\end{definition}

Let 
$$R_{1,\infty}:=\max_{\bL\in\Slinf}\exp\left(\sum_{i=1}^nb_i^1|\log|\lambda_i|_\bL|\right).$$

\begin{lemma}\label{LemmaInitialLC2}
Every pair of solutions $(\lambda,\mu)$ lies in $\mathcal L^2_\infty(B_1,C_1,R_{1,\infty})$.
\end{lemma}

\begin{proof}
For a $\bL\in\Slinf$ we have,
\begin{align*}
|\log|\lambda|_\bL|&\leq\sum_{i=1}^nx_i|\log|\lambda_i|_\bL|\leq\sum_{i=1}^nb_i^1|\log|\lambda_i|_ 
\bL|\\
\leq&\max_{\bL\in\Slinf}\sum_{i=1}^nb_i^1|\log|\lambda_i|_\bL|=\log(R_{1,\infty}).
\end{align*}
\end{proof}

\begin{definition}
Let $B$ and $C$ be as above then for $\bL\in\Slinf$ and $1< R^\prime<R$ we define,
$$T^2_\bL(B,C,R,R^\prime)=\biggl\{ (\lambda,\mu)\in\mathcal L^2_\infty(B,C,R):\left.
\begin{array}{c}
|\mu-1|_\bL<\frac{1}{R^\prime}\text{ or }\\
|\frac{\mu}{\lambda}-1|_\bL <\frac{1}{R^\prime}
\end{array}\right\rbrace.$$
\end{definition}

We need the following lemma,

\begin{lemma}\label{Lemmac1}
There is a computable constant $\cinf >0$ such that 
$$x_i\leq\cinf\max_{\bL\in\Slinf}(|\log|\linf|_\bL|)$$
for all $i\in I^\infty$.
\end{lemma}

\begin{proof}
We may assume that $I^\infty=\{1,2,\cdots,t\}$ and $\Slinf=\{\bL_1,\cdots,\bL_u\}$ with $t\leq u$. We define the matrix 
$$M=
\begin{pmatrix}
\log|\lambda_1|_{\bL_1} & \cdots & \log|\lambda_t|_{\bL_1}\\
\vdots & & \vdots\\
\log|\lambda_1|_{\bL_u} & \cdots & \log|\lambda_t|_{\bL_u}
\end{pmatrix}
$$
By the choice of $I^\infty$ and $\Slinf$ there exists a $t\times t$ submatrix of $M$ which 
is invertible. Among all these submatrices we pick one, which we call $M_t$, whose inverse has 
the maximal infinity norm. Define $\cinf=\Vert M_t^{-1}\Vert_\infty$. Then we can deduce that 
$|x_i|\leq\cinf\max_{\bL\in\Slinf}(|\log|\linf|_\bL|)$.
\end{proof}

\begin{proposition}\label{PropSplitSolutionsC2}
Let $1< R_{k+1}<R_k$, and let $B_k$ and $C_k$ be vectors of the exponent bounds such that every solution $(\lambda,\mu)$ lies in $\mathcal L^2_\infty(B_k,C_k,R_k)$. Then,
\begin{align*}
\mathcal L^2_\infty(B_k,C_k,R_k)=&\mathcal L^2_\infty(B_{k+1},C_{k+1},R_{k+1})\bigcup\\
&\bigcup_{\bL\in\Slinf}T^2_\bL(B_k,C_k,R_k,R_{k+1})
\end{align*}
where $C_{k+1}=C_k$, $b_i^{k+1}=\min(b_i^k,\cinf\log (R_{k+1})+\cinf c_{2,\infty})$ for $i\in 
I^\infty$, otherwise $b_i^{k+1}=b_i^k$ and $c_{2,\infty}=\max_{\bL\in\Slinf}\sum_{i\not\in I^\infty} 
b_i^k|\log|\lambda_i|_\bL|$.
\end{proposition}

\begin{proof}
Let $(\lambda,\mu)\in\mathcal L^2_\infty(B_k,C_k,R_k)$ but $(\lambda,\mu)\not\in\mathcal L^2_\infty 
(B_k,C_k,R_{k+1})$. That means there exists $\bL\in\Slinf$ such that $|\lambda|_\bL>R_{k+1}$ or 
$|\lambda|_\bL<\frac{1}{R_{k+1}}$. In the first case we get that,
\begin{align*}
|\lambda|_\bL&>R_{k+1}\Leftrightarrow |\tau(\lambda)|_{\tau(\bL)}>R_{k+1}\Leftrightarrow\\
|\frac{1}{\lambda}|_{\tau(\bL)}&<\frac{1}{R_{k+1}}\Leftrightarrow|\frac{\mu}{\lambda}-1|_{\tau(\bL)} 
<R_{k+1}.
\end{align*}
In the second case we get $\mid\lambda\mid_\bL<\frac{1}{R_{k+1}}\Leftrightarrow\mid\mu-1\mid_\bL< 
\frac{1}{R_{k+1}}$. Finally, $(\lambda,\mu)\in T^2_\bL(B_k,C_k,R_k,R_{k+1})$ or $T^2_{\tau(\bL)}
(B_k,C_k,R_k,R_{k+1})$.

Now for $(\lambda,\mu)\in\mathcal L^2_\infty(B_k,C_k,R_{k+1})$ we have that 
$$|\log|\linf|_\bL|\leq |\log|\lambda|_\bL|+|\log|\frac{\lambda}{\linf}|_\bL|<\log(R_{k+1})+c_{2,
\infty}$$
and by Lemma (\ref{Lemmac1}) we get,
$$x_i\leq\cinf\log (R_{k+1})+\cinf c_{2,\infty}.$$
So, we deduce that $(\lambda,\mu)\in\mathcal L^2_\infty(B_{k+1},C_{k+1},R_{k+1})$. 
\end{proof}

Proposition \ref{PropSplitSolutionsC2} is very useful in practice because we can quickly prove when 
the set $T^2_\bL(B_k,C_k,R_k, R_{k+1})$ has non--trivial solutions or not. In paragraph 3.1 Lemma 3 
of \cite{Smart99} Smart shows how to do that. The only difference we have introduced, which does not 
change the construction, is that we allow different upper bounds for the exponents while they have 
for all the same bound. We leave to the reader to see how the proof of Lemma 3 in \cite{Smart99} 
adapts to our case.

\paragraph*{$C_3$ case}

Let $\sigma$ be the generator of $\Gal(L/K)$. We recall that we have chosen bases of $\Gl$ and $\Gm$ 
such that $n=m$ and $x_i=y_i$ for all $i=0,\cdots, n$ according to Theorem \ref{TheoremC3} by 
choosing $\mu_i=\sigma(\frac{1}{\lambda_i})$. So, we have to consider only one bound vector $B_k$ at each step of the sieve. Also, we want to recall that $\Gl=\Gm$ and as a result $\Sl=\Sm$. By Theorem 
\ref{TheoremC3} we see that $\Sl$ contains only split finite primes.

Again, as in the quadratic case, the first step is to find an upper bound on $|\ord_\bL(\lambda)|$ for each $\bL\in\Sl$, by proving that, 
$$|\mu-1|_\bL <\delta\ll 1$$
has no non--trivial solutions. We use the new upper bounds of $|\ord_\bL(\lambda)|$ to find a new 
bound vector $B_1$. 

Now, we want to reduce the bound for the unit generators. We observe that $I^\infty=J^\infty$.

\begin{definition}
Let $B=(b_0,b_1,\cdots,b_n)\in\NN^{n+1}$. Then for $R>1$ we define,
$$\mathcal L^3_\infty(B,R)=\left\lbrace(\lambda,\mu):|x_i|\leq b_i`\text{ and }|\log|\lambda|_\bL|
\leq \log(R),\forall\bL\in\Slinf\right\rbrace.$$
\end{definition}

%If $R_{1,\infty}$ is as above,

\begin{lemma}
Every pair of solutions $(\lambda,\mu)$ lies in $\mathcal L^3_\infty(B_1,R_{1,\infty})$ with $R_{1,
\infty}$ as above.
\end{lemma}

\begin{proof}
Similar to the quadratic case, Lemma \ref{LemmaInitialLC2}.
\end{proof}

\begin{definition} 
Let $B$ be as above, then for each $\bL\in\Slinf$ and $1< R^\prime<R$ we define,
$$T^3_\bL(B,R,R^\prime)=\biggl\{ (\lambda,\mu)\in\mathcal L^3_\infty(B,R):\left.
\begin{array}{c}
|\mu-1|_\bL<\frac{1}{R^\prime}\text{ or }\\
|\lambda-1|_{\sigma(\bL)} <\frac{1}{R^\prime}
\end{array}\right\rbrace.$$
\end{definition}

If $\cinf$ is as in Lemma \ref{Lemmac1} then, we have:

\begin{proposition}\label{PropSplitSolutionsC3}
Let $1< R_{k+1}<R_k$ and $B_k$ be bound vector of the exponents. Then,
$$\mathcal L^3_\infty(B_k,R_k)=\mathcal L^3_\infty(B_{k+1},R_{k+1})\cup\bigcup_{\bL\in\Slinf}
T^3_\bL(B_k,R_k,R_{k+1})$$
where $b_i^{k+1}=\min(b_i^k,\cinf\log (R_{k+1})+\cinf c_{2,\infty})$ for $i\in 
I^\infty$ otherwise $b_i^{k+1}=b_i^k$ and $c_{2,\infty}=\max_{\bL\in\Slinf}\sum_{i\not\in I^\infty} 
b_i^k |\log|\lambda_i|_\bL|$.
\end{proposition}

\begin{proof}
Similar to Proposition \ref{PropSplitSolutionsC2}.
\end{proof}

%In the final simple loop we do at the end we also use general Hilbert symbol to reduce the number of cases we have to check. For our solutions $(\lambda,\mu)$ always hold $\left(\frac{\lambda,\mu}{\bL} \right)_\ell$ for every prime $\bL$ and integer $\ell\geq 1$. For a suitable choice\footnote{We use tame Hilbert symbol because they are formulas to evaluate it and are easy to be implemented compare to the general case. That means that we choice only primes in $\Sl\cup\Sm$ and $\ell$ is the order of the unit group of the residue field. For details about Hilbert symbol see \cite{Neukirch86}.} of $\bL$ and $\ell$ we create the matrix $A_\bL=\left(\left(\frac{\lambda_i,\mu_j}{\bL}\right)_\ell \right)_{i,j}\in\mathbb{M}_{n+1,n+1}(\ZZ/\ell\ZZ)$ and we ask $xA_\bL x^t=0$.

\paragraph*{$S_3$ case}

As above let $\Gal(L/K)=\langle\sigma,\tau\rangle$ where $\sigma^3=\tau^2=1$ and 
$\tau\sigma\tau=\sigma^2$. As in the cubic case we may choose bases of $\Gl$ and $\Gm$ such that 
$n=m$ and $x_i=y_i$ for all $i=0,\cdots, n$ using Theorem \ref{TheoremS3} and choosing 
$\mu_i=\sigma(\frac{1}{\lambda_i})$. Again, we have only to consider only one bound vector $B_k$ and 
we have $I^\infty=J^\infty$. However, we may now have $\Gl\neq\Gm$ and $\Sl\neq\Sm$. We define 
$\Sminf=S_{J^\infty}$.

Again, the first step is to find an upper bound of $|\ord_\bL(\lambda)|$ for each $\bL\in\Sl$ by 
proving that, 
$$|\mu-1|_\bL <\delta\ll 1$$
has no non--trivial solutions. We use the new upper bounds of $|\ord_\bL(\lambda)|$ to get new bounds 
$B_1$. 

\begin{definition}
Let $B=(b_0,b_1,\cdots,b_n)\in\NN^{n+1}$. Then for $R>1$ we define,
$$\mathcal L^6_\infty(B,R)=\left\lbrace(\lambda,\mu):|x_i|\leq b_i\text{ and }
\begin{array}{c}
|\log|\lambda|_\bL|\leq \log(R)\\
|\log|\mu|_\bL|\leq \log(R)
\end{array}, \forall\bL\in\Slinf\cup\Sminf\right\rbrace.$$
\end{definition}

%By symmetry, we see that we get the same constant $R_{1,\infty}$ for both $\lambda$ and $\mu$. The same holds for the constant $\cinf$ in lemma \ref{Lemmac1}.

\begin{lemma}
Every pair of solutions $(\lambda,\mu)$ lies in $\mathcal L^6_\infty(B_1,R_{1,\infty})$ with $\cinf$ 
and $R_{1,\infty}$ as above.
\end{lemma}

\begin{proof}
Similar to lemma \ref{LemmaInitialLC2}.
\end{proof}

\begin{definition} 
Let $B$ be as above. Then for each $\bL\in\Slinf\cup\Sminf$ and $1< R^\prime<R$ we define,
$$T^6_\bL(B,R,R^\prime)=\biggl\{ (\lambda,\mu)\in\mathcal L^6_\infty(B,R):\left.
\begin{array}{c}
|\mu-1|_\bL<\frac{1}{R^\prime}\text{ or }\\
|\lambda-1|_{\bL} <\frac{1}{R^\prime}\text{ or }\\
|\mu-1|_{\sigma^2(\bL)}<\frac{1}{R^\prime}\text{ or }\\
|\lambda-1|_{\sigma(\bL)}<\frac{1}{R^\prime}
\end{array}\right\rbrace.$$
\end{definition}

The following proposition is proved in the same way as Proposition \ref{PropSplitSolutionsC3},

\begin{proposition}\label{PropSplitSolutionsS3}
Let $1< R_{k+1}<R_k$ and $B_k$ be bound vector of the exponents. Then it holds,
$$\mathcal L^6_\infty(B_k,R_k)=\mathcal L^6_\infty(B_{k+1},R_{k+1})\cup \bigcup_{\bL\in\Slinf\cup 
\Sminf}T^6_\bL(B_k,R_k,R_{k+1})$$
where $b_i^{k+1}=\min(b_i^k,\cinf\log (R_{k+1})+\cinf c_{2,\infty})$ for $i\in 
I^\infty$ otherwise $b_i^{k+1}=b_i^k$ and $c_{2,\infty}=\max_{\bL\in\Slinf}\sum_{i\not\in I^\infty} 
b_i^k |\log|\lambda_i|_\bL|$.
%=\max_{\bL\in\Sminf}\sum_{i\not\in J^\infty}b_i^k |\log|\mu_i|_\bL|$.
\end{proposition}

%\begin{proof}
%The proof is similar to proposition \ref{PropSplitSolutionsC2}.
%\end{proof}

We should mention that in this case it may happen that there exists $\bL\in\Slinf\cup\Sminf$ such 
that $\ord_\bL(\lambda)=0$. Then we can not use Lemma 3 in \cite{Smart99} to prove that the 
inequality $|\lambda-1|_\bL<\frac{1}{R_{k+1}}$ does not contain non--trivial solutions. However, for 
such a solution we have $|\log|\lambda|_\bL|=0$ and $|\log|\mu|_\bL|\leq\log(R_{k+1})$ and if it does 
not lie in any set $T^6_\bL(B_k,R_k,R_{k+1})$ then it has to be in $L^6_\infty(B_{k+1},R_{k+1})$.

If we continue this way we reach to a point where we are not able to prove that the sets $T^i_\bL$ do 
not contain non-trivial solutions. If the bounds and the rank of the groups are small we can just do 
a simple loop to find all solutions, using these bounds. However, we may still have a lot of cases to 
check. The idea is to find all the solutions such that $\lambda\equiv 1\mod\bL^e$ for all finite 
primes in $\bL\in\Sm$ and suitable(small) choice of $e$. Now, we know that the remaining solutions 
have smaller valuation for all finite primes, and we can deduce smaller bounds for the non--unit 
generators of $\Gl$ and $\Gm$. Since we have smaller upper bounds for the non--unit generators, we 
can apply all the above results to reduce smaller bounds for the unit generators. 

We can repeat the above procedure up to the point where a simple loop is feasible. In the $C_3$ and 
$S_3$ cases we also use general Hilbert symbols to reduce the number of cases we have to check in the 
final loop. Our solutions $(\lambda,\mu)$ always satisfy $\left(\frac{\lambda,\mu}{\bL} \right)_\ell$ 
for every prime $\bL$ and integer $\ell\geq 1$. For a suitable choice\footnote{We use tame Hilbert 
symbols because there are formulas to evaluate it which are easy implement compared to the general 
case. That means we choose only primes in $\Sl\cup\Sm$ and choose $\ell$ to be the order of the unit 
group of the residue field at $\bL$.} of $\bL$ and $\ell$ we create the matrix $A_\bL=\left(\left( 
\frac{\lambda_i,\mu_j}{\bL}\right)_\ell \right)_{i,j}\in \mathbb{M}_{n+1,n+1}(\ZZ/\ell\ZZ)$ and we 
test whether $xA_\bL x^t=0$. 

In practice, it seems that we have to solve $\lambda\equiv 1\mod\bL^e$ just once. However, there are 
cases where trying to find the solution of $\lambda\equiv 1\mod\bL^e$ it is the most expensive part.

\section{Examples}

For reasons of space we cannot include the full details of the results obtained here; complete lists 
of the join variants in each case may be found at the author's web page 
\textit{http://www2.warwick.ac.uk/fac/sci/maths/people/staff/koutsianas}.

We have compared our results with Cremona's and LMFDB's database\cite{LMFDB}. Finally, we have to say 
that all the computations have been done in the servers at Mathematics Institute of University of 
Warwick and all the code has been written in Sage\cite{Sage}. 

\subsection{$K=\QQ$}

As in \cite{CremonaLingham07} we have computed all the curves for set of primes $S=\{2\},\{2,3\}$ and 
$\{2,p\}$ for $p=3,\cdots, 23$. We present a couple of examples in order to compare our method with 
the Cremona--Linghm method. We also give an example with $S=\{2,3,23\}$ which seems to be beyond the 
range of the Cremona--Lingham method.

\paragraph{Case $S=\{2,3\}$}

We found 83 $j$--invariants as in \cite{CremonaLingham07}. We solved 3 $S$--unit equations for the 
$C_2$ case, one $S$--unit equation for the $C_3$ case and 8 $S$--unit equations for the $S_3$. In all 
these cases we had no difficulties in the reduction and sieve steps. However, it seems that the 
Cremona--Lingham method works faster because we have the large number of $S$--unit equations to solve 
in the $S_3$ case.

\paragraph{Case $S=\{2,17\}$} 

For this set of primes we did not have the difficulties as in \cite{CremonaLingham07} where they had 
to evaluate generators of the Mordell-Weil group with huge denominators of the x--coordinate using 
Heegner points. We only had to solve 3 $S$--unit equations over the quadratic fields 
$\QQ(\sqrt{17})$, $\QQ(\sqrt{2})$ and $\QQ(\sqrt{34})$. We found 29 $j$--invariants.

It would be very useful if we could benefit in Cremona--Lingham method by the fact that we had to 
solve $S$--unit equations only for these 3 quadratic fields. So far, we can not see how we could 
combine the two methods.

\paragraph{Case $S=\{2,3,23\}$} We had to solve $S$--unit equations for 8 quadratic fields, 1 cubic 
extension and 37 $S_3$ extensions. We found 311 isomorphism classes of curves and 5504 curves. We got 
curves with the possible maximal conductor $2^83^523^2=32908032$ which is beyond the range of the 
current modular symbol method \cite{CremonaBook}. Moreover, in Cremona--Lingham method there were 
cases where the Mordell--Weil groups could not be computed. It worthy to mention that the most 
expensive part of the computation was the sieve for the only one $S_3$ case with $\rank(\Gl)=4$. It 
took around 6 times more than all the other computations we had.

\paragraph{More computations} Apart from the sets of primes $S$ we presented above, we have computed 
all the curves for the following set of primes,
\begin{itemize}
\item Curves with at least one rational 2--torsion point for the set of primes $S=\{p\}$ for $p\leq 
200$, $S=\{p,q\}$ for $p\leq 11$ and $q\leq 200$, $S=\{2,3,p\}$ for $p\leq 43$ and $S=\{2,5,p\}$ for 
$p\leq 37$.

\item Curves with cubic 2--division field for $S=\{p\}$ for $p\leq 200$, $S=\{p.q\}$ for $p\leq 5$ 
and $q\leq 200$, $S=\{7,p\}$ for $p\leq 29$ and $S=\{2,3,p\}$ for $p\leq 41$.

\item Curves with $S_3$ 2--division field for $S=\{p\}$ for $p\leq 181$ and $S=\{2,p\}$ for $p\leq 
127$.
\end{itemize}

From the computations we have done and the current implementation, it seems that the method works 
well when the set $S$ does not have very many number of primes and the primes are not too big. For 
example, it seems impossible to us to do computations with primes with 4, 5 or more digits. Even the 
calculation of $S$--unit groups over quadratic fields becomes difficult in this case.

\subsection{Quadratic fields}

\subsubsection{$K=\QQ(\sqrt{-1})$}

%We use LMFDB label to denote $K=\QQ(\sqrt{-1})$. So, let $i$ be the root of the polynomial $x^2+1$. 

\paragraph{Case $S=\{\pK |2\}$} We found 17 $j$--invariants and 64 curves. We had to solve 3 
$S$--unit equations for quadratic extensions of $K$. We get the same number of $j$--invariants and 
curves as in \cite{Laska83}.

%\paragraph{Case $S=\{\pK |6\}$}

%As in \cite{Laska83} and LMFDB database we have computed 

%\subsubsection{$K=\QQ(\sqrt{-3})$}

%Let $\theta$ be the root of the polynomial $f(x)=x^2-x+1$.

%\paragraph{Case $S=\{\pK |2\}$} We found 6 $j$--invariants and 24 curves. We had to solve 3 
%$S$--unit equations for quadratic extensions of $K$. 

\subsubsection{$K=\QQ(\sqrt{5})$}

\paragraph{Case $S=\{\pK|31\}$} For this case we only have partial results. We have computed all the 
curves with at least one rational 2--torsion point. We have found 12 $j$--invariants and 24 curves. 
We solved $S$--unit equations for 7 quadratic extensions of $K$.

\bibliographystyle{alpha}
\bibliography{Elliptic_curves_with_good_reduction_paper}

\end{document}